\documentclass{amsart}
\usepackage{graphicx}
\usepackage{amssymb}
\usepackage{epstopdf}
\usepackage{amscd}
\usepackage{amsthm}
\usepackage[all, knot]{xy}
\usepackage{extpfeil}
\xyoption{all}
\xyoption{arc}
\usepackage[hidelinks]{hyperref}
\usepackage{amsmath}
\def\cP{\mathcal P}

\def\cV{\mathcal V}

\def\a{\mathfrak{a}}
\def\b{\mathfrak{b}}
\def\cc{\mathfrak{c}}
\def\m{\mathfrak{m}}

\def\p{\mathfrak{p}}
\def\q{\mathfrak{q}}

\def\L{\mathbf{L}}

\newcommand{\Z}{\mathbb{Z}}

\def\RHom{\operatorname{{\mathbf R}Hom}}

\def\LLambda{\operatorname{{\mathbf L}\Lambda}}

\def\ker{\operatorname{ker}}

\def\pd{\operatorname{pd}}

\def\invlim{\varprojlim}

\def\del{\partial}

\def\Spec{\operatorname{Spec}}

\def\supp{\operatorname{supp}}
\def\cosupp{\operatorname{cosupp}}
\def\Hom{\operatorname{Hom}}

\def\Ext{\operatorname{Ext}}

\def\cone{\operatorname{cone}}

\def\K{\operatorname{K}}

\def\H{\operatorname{H}}
\def\D{\operatorname{D}}

\newcommand{\PurInj}{\mathcal{PI}}

\numberwithin{equation}{section}
\newcounter{intro}
\theoremstyle{plain} %% This is the default, anyway
\newtheorem{thm}[equation]{Theorem}

\newtheorem{question}[equation]{Question}
\newtheorem{cor}[equation]{Corollary}
\newtheorem{lem}[equation]{Lemma}
\newtheorem{prop}[equation]{Proposition}

\newtheorem{thmintro}[intro]{Theorem}

\newtheorem{propintro}[intro]{Proposition}

\theoremstyle{definition}

\newtheorem{example}[equation]{Example}

\theoremstyle{remark}
\newtheorem{rem}[equation]{Remark}

\title[Cosupport computations for finitely generated modules]{Cosupport computations for finitely generated modules over commutative noetherian rings} 
\author{Peder Thompson}
\address{Department of Mathematics and Statistics\\ Texas Tech University\\ Broadway and Boston\\ Lubbock, TX 79409}
\email{peder.thompson@ttu.edu}
\date{\today}                                           % Activate to display a given date or no date
\subjclass[2010]{13D02, 13D07, 13D09, 13C11, 13E05}
\thanks{{\em Key words and phrases:}
cosupport, minimal complex, cotorsion flat module, countable ring}

\begin{document}
\begin{abstract}
We show that the cosupport of a commutative noetherian ring is precisely the set of primes appearing in a minimal pure-injective resolution of the ring. As an application of this, we prove that every countable commutative noetherian ring has full cosupport. We also settle the comparison of cosupport and support of finitely generated modules over any commutative noetherian ring of finite Krull dimension. Finally, we give an example showing that the cosupport of a finitely generated module need not be a closed subset of $\Spec R$, providing a negative answer to a question of Sather-Wagstaff and Wicklein \cite{SWW17}.
\end{abstract}
\maketitle

%\tableofcontents

%%%%%%%%%%%%%%%%%%%%%%%%%%%%%%%%%%%%%%%%
\section*{Introduction}
The theory of cosupport, recently developed by Benson, Iyengar, and Krause \cite{BIK12} in the context of triangulated categories, was partially motivated by work of Neeman \cite{Nee11}, who classified the colocalizing subcategories of the derived category of a commutative noetherian ring. Despite the many ways in which cosupport is dual to the more established notion of support introduced by Foxby \cite{Fox79,BIK08}, cosupport seems to be more elusive, even in the setting of a commutative noetherian ring. Indeed, the supply of finitely generated modules for which cosupport computations exist is limited. One purpose of this paper is to provide such computations. 

We first show that for a finitely generated module over a commutative noetherian ring of finite Krull dimension, its cosupport is the intersection of its support and the cosupport of the ring, which places emphasis on computing the cosupport of the ring itself. With this in mind, we prove that countable commutative noetherian rings have full cosupport, and hence cosupport and support coincide for finitely generated modules over such rings having finite Krull dimension.  We also give new examples of uncountable rings that have full cosupport.  Finally, we present an example of a ring whose cosupport is not closed, unlike support, yielding a negative answer to a question posed by Sather-Wagstaff and Wicklein \cite{SWW17}. 

One method to determine the support of a module is to identify primes appearing in its minimal injective resolution, as done by Foxby \cite{Fox79}, using the decomposition of injective modules described by Matlis \cite{Mat58}. Our systematic approach to computing cosupport is to appeal to the parallel decomposition of cotorsion flat modules due to Enochs \cite{Eno84} and use minimal cotorsion flat resolutions studied in \cite{Tho17a}.

\begin{center}* \quad * \quad *\end{center}

Our goal is to better understand cosupport in the setting of a commutative noetherian ring.  Over such a ring $R$, the cosupport of a complex $M$ is denoted $\cosupp_RM$.  This is the set of prime ideals $\p$ such that $\RHom_R(R_\p,\LLambda^{\p}M)$ is not acyclic, where $\LLambda^\p(-)$ is left derived $\p$-adic completion; see Section \ref{cosupp_section} for details.  Prompted by the fact that if $M$ is a finitely generated $\Z$-module, then there is an equality $\cosupp_\Z M=\supp_\Z M$ \cite[Proposition 4.18]{BIK12}, we investigate to what extent cosupport and support agree for finitely generated modules. The cosupport of finitely generated modules over a 1-dimensional domain having a dualizing complex is known \cite[Theorem 6.11]{SWW17}; this is recovered by part (2) of the following. Part (3) gives an affirmative answer to a question in \cite{SWW17}.

\begin{thmintro}[cf. Theorem \ref{full_cosupp_rings}, Corollary \ref{cosupp_equal_supp}]\label{countable_rings_thm_intro}
Let $R$ be one of the following:
\begin{enumerate}
\item A countable commutative noetherian ring;
\item A finite ring extension of a 1-dimensional commutative noetherian domain that is not complete local;
\item A finite ring extension of $k[x,y]_{(x,y)}$ for any field $k$.
\end{enumerate}
Then $R$ has full cosupport, i.e., $\cosupp_R R=\Spec R$. 

If $R$ is one of these rings and has finite Krull dimension, and $M$ is an $R$-complex with degreewise finitely generated cohomology, then $\cosupp_RM=\supp_RM$.
\end{thmintro}

An obstruction to having full cosupport is completeness at a non-zero ideal. In particular, if $(R,\m)$ is a complete local ring, then $\cosupp_RR=\{\m\}$. Setting $\cc_R$ to be the largest ideal of $R$ such that $R$ is $\cc_R$-complete, the inclusion 
\begin{align*}
\cosupp_RR \subseteq \cV(\cc_R) \tag{$\star$}
\end{align*}
always holds \cite[Proposition 4.19]{BIK12} (here, $\cV(\cc_R)=\{\p\in \Spec R\mid \p\supseteq \cc_R\}$). 
We give a condition for equality to hold in ($\star$): 
\begin{propintro}[cf. Proposition \ref{cosupp_complete_closed_thm}]  
Let $R$ be a commutative noetherian ring and $\cc_R$ be as above. Then $\cosupp_RR=\cV(\cc_R)$ if and only if $R/\cc_R$ has full cosupport.
\end{propintro}
For example, if $R$ is a ring such that Theorem \ref{countable_rings_thm_intro} applies to $R/\cc_R$, then equality of ($\star$) holds. 
However, Example \ref{cosupp_not_closed} shows that strict inequality in ($\star$) can occur, providing a negative answer to a question in \cite{SWW17}; moreover, this example shows that $\cosupp_RR$ need not be a closed subset of $\Spec R$, unlike the support of a finitely generated module.

Two of our main results towards establishing these goals are Theorem \ref{cosupp_cf} and Theorem \ref{finite_map_cosupp}, which involve the notion of cotorsion flat modules (recalled in Section \ref{cosupp_detected_by_primes_subsection}).  In Theorem \ref{cosupp_cf}, we describe how the cosupport of a module can be computed by identifying the prime ideals appearing in a certain minimal complex of cotorsion flat modules. Further, in Theorem \ref{finite_map_cosupp} we show how cosupport passes along finite ring maps, by explicitly examining the structure of a minimal cotorsion flat resolution of the ring.

\begin{center}* \quad * \quad *\end{center}

An outline of the paper: We set notation and define cosupport for our setting in Section \ref{cosupp_section}, and in Section \ref{cosupp_detected_by_primes_subsection} we show how cosupport can be detected by minimal complexes of cotorsion flat modules (Theorem \ref{cosupp_cf}). In Section \ref{section_cotorsion} we compute cosupport of cotorsion modules. In Section \ref{cosupp_of_fg_modules}, we compare cosupport and support of finitely generated modules (Corollary \ref{cosupp_vs_supp}), prove a result describing how cosupport passes along finite ring maps (Theorem \ref{finite_map_cosupp}), and establish Theorem \ref{countable_rings_thm_intro} from above.  Finally, in Section \ref{examples_section} we give a number of explicit examples of cosupport of commutative noetherian rings, including an example that exhibits a ring without closed cosupport (Example \ref{cosupp_not_closed}).

%%%%%%%%%%%%%%%%%%%%%%%%%%%%%%%%%%%%%%%%
\section{Cosupport in a commutative noetherian ring}\label{cosupp_section}
We set notation, discuss certain derived functors, and define cosupport.

\subsection*{Setting and notation}
Let $R$ be a commutative noetherian ring throughout this paper. Our main objects of study are complexes of $R$-modules, primarily in the derived category $\D(R)$, which we now briefly describe. 

A {\em complex of $R$-modules}, or {\em $R$-complex} for short, is a $\Z$-graded $R$-module along with a differential whose square is zero. An $R$-complex $C$, whose differential is understood to be $\del_C$, is written as
$$\cdots \xrightarrow{\del_C^{i-1}} C^i \xrightarrow{\del_C^i} C^{i+1}\xrightarrow{\del_C^{i+1}}\cdots,$$
where we primarily index cohomologically. We say that an $R$-complex $C$ is {\em bounded on the left (respectively, right)} if $C^i=0$ for $i\ll0$ (respectively, $C^i=0$ for $i\gg0$).  For $R$-complexes $C$ and $D$, the total tensor product complex $C\otimes_R D$ and total Hom complex $\Hom_R(C,D)$ are defined as direct sum and direct product totalizations of their corresponding double complexes, respectively.  An $R$-complex $C$ is {\em acyclic} if $\H^i(C)=0$ for all $i\in \Z$.

The {\em homotopy category} $\K(R)$ is the category whose objects are $R$-complexes and morphisms are degree zero chain maps up to chain homotopy.  If we also invert all quasi-isomorphisms (morphisms of $R$-complexes which induce an isomorphism on cohomology), we obtain the {\em derived category} of $R$, denoted $\D(R)$.  We use the symbol $\simeq$ to denote isomorphisms in $\D(R)$ (i.e., to indicate there is a diagram of quasi-isomorphisms between two complexes). For details on complexes, homotopies, and the derived category, see for example \cite{Avr98} or \cite[Chapter 10]{Wei94}.

We say an $R$-complex $F$ is {\em semi-flat} if $F^i$ is flat for $i\in \Z$ and $F\otimes_R-$ preserves quasi-isomorphisms.  An $R$-complex $P$ is {\em semi-projective} if $P^i$ is projective for $i\in \Z$ and $\Hom_R(P,-)$ preserves quasi-isomorphisms. Dually, an $R$-complex $I$ is {\em semi-injective} if $I^i$ is injective for $i\in \Z$ and $\Hom_R(-,I)$ preserves quasi-isomorphisms. (These are the ``DG-flat/projective/injective'' complexes of \cite{AF91}.) Every $R$-complex $M$ has a {\em semi-projective resolution} (and hence also a {\em semi-flat resolution}) $F\xrightarrow{\simeq}M$ (for existence of such resolutions, see \cite[Proposition 5.6]{Spa88} and also \cite[1.6]{AF91}); similarly, semi-injective resolutions exist. This extends the classical notions of projective, flat, and injective resolutions of modules. 

\subsection*{Derived completion and colocalization}
We remind the reader of two functors on $\D(R)$ that will be used to define cosupport: left-derived completion and right-derived colocalization.  For an ideal $\a\subset R$, the $\a$-adic completion of an $R$-module $M$ is defined as $\Lambda^\a M=\invlim_n(R/\a^n\otimes_R M)$; it will also be denoted by $\widehat{M}^\a$.  This extends to a functor on the homotopy category $\Lambda^\a:\K(R)\to \K(R)$ and---being not necessarily exact outside of the category of finitely generated $R$-modules---has a left derived functor $\LLambda^\a:\D(R)\to \D(R)$, defined using semi-projective resolutions; see \cite{AJL97} and \cite{PSY14}.  For any $R$-complex $M$, and semi-flat resolution $F\xrightarrow{\simeq}M$ (or, more generally, a semi-flat complex $F$ isomorphic to $M$ in $\D(R)$), we have an isomorphism in $\D(R)$ \cite[page 31]{Lip02}; see also \cite[Proposition 3.6]{PSY14}:
\begin{align}\label{Lambda_semiflat}
\LLambda^{\p}M\simeq \invlim_n(R/\p^n\otimes_R F).
\end{align}

The functor $\RHom_R(R_\p,-):\D(R)\to \D(R)$, referred to as right-derived colocalization, is the usual right derived functor of $\Hom_R(R_\p,-)$; namely, letting $M\xrightarrow{\simeq} I$ be any semi-injective resolution of $M$, there is an isomorphism in $\D(R)$: 
$$\RHom_R(R_\p,M)\simeq \Hom_R(R_\p,I).$$ 

\subsection*{Cosupport}
For an $R$-complex $M$, we define (as in \cite{BIK12}) the {\em cosupport} of $M$ to be 
\begin{align}\label{cosupp_defn}
\cosupp_RM=\{\p\in \Spec R\mid  \RHom_R(R_\p,\LLambda^{\p}M)\not\simeq0\}.
\end{align}
This agrees with a variety of other ways to define cosupport; for example, setting $\kappa(\p)=R_\p/\p R_\p$ to be the residue field of $R_\p$, we have \cite[Proposition 4.4]{SWW17}: 
\begin{align}\label{equiv_defns_of_cosupp}
\p\in \cosupp_RM  \Leftrightarrow \RHom_R(\kappa(\p),M)\not\simeq 0 \Leftrightarrow \kappa(\p)\otimes_{R_\p}^\L \RHom_R(R_\p,M)\not\simeq 0.
\end{align} 
In particular, in view of (\ref{Lambda_semiflat}), for a flat $R$-module $F$,
\begin{align}\label{cosupp_for_flat_module_defn}
\cosupp_RF=\{\p\in \Spec R\mid \Ext_R^*(R_\p,\Lambda^\p F)\not=0\}.
\end{align}

For comparison, the {\em support}\footnote{This is referred to as the {\em small support} in \cite{Fox79}.} of a complex $M$ with $\H^i(M)=0$ for $i\ll0$ is:
$$\supp_RM=\{\p\in \Spec R\mid \kappa(\p)\otimes_R^\L M\not\simeq0\}.$$
Equivalently, $\supp_RM$ is the set of prime ideals such that $E(R/\p)$ occurs as one of the indecomposable injective modules in the minimal semi-injective resolution of $M$ \cite[Remark 2.9]{Fox79}, see also \cite{CI10}. This perspective motivates the main result in the next section.

%%%%%%%%%%%%%%%%%%%%%%%%%%%%%%%%%%%%%%%%
\section{Cosupport via minimal complexes of cotorsion flat modules}\label{cosupp_detected_by_primes_subsection}
We show in this section that minimal cotorsion flat complexes, as characterized in \cite{Tho17a}, can be used to detect cosupport.

\subsection*{Cotorsion flat modules}
An $R$-module $T$ is called {\em cotorsion flat} if it is flat and satisfies $\Ext_R^1(F,T)=0$ for every flat $R$-module $F$ (i.e., it is also cotorsion). Enochs showed \cite[Theorem]{Eno84} that cotorsion flat $R$-modules decompose uniquely as a product of completions of free $R_\q$-modules, for $\q\in \Spec R$; namely, $T$ is cotorsion flat if and only if 
\begin{align}\label{CF_decomposition}
T\cong \prod_{\q\in \Spec R}\widehat{R_\q^{(X_\q)}}^\q,
\end{align}
for some (possibly empty or infinite) sets $X_\q$.   
Set $T_\q=\widehat{R_\q^{(X_\q)}}^\q$ (despite the notation, we caution this is not a localization of $T$, rather it is the $q$-th component of $T$). For a cotorsion flat $R$-module $T$ as in (\ref{CF_decomposition}) and a fixed prime ideal $\p$, there are isomorphisms \cite[Lemma 2.2]{Tho17a}:
\begin{align}\label{completion_and_colocalization_of_CF}
\Lambda^\p(T)\cong \prod_{\q\supseteq \p} T_\q \quad \text{ and }\quad \Hom_R(R_\p,T)\cong \prod_{\q\subseteq \p} T_\q.
\end{align}
Further, for any complex $B$ of cotorsion flat $R$-modules, the natural maps $R\to R_\p$ and $R\to \Lambda^\p R$ induce degreewise split maps: $\Hom_R(R_\p,B)\hookrightarrow B$ and $B\twoheadrightarrow \Lambda^\p B$. 

For an indecomposable injective $R$-module $E(R/\p)$ and set $X$, there is an isomorphism \cite[Lemma 4.1.5]{Xu96}:
\begin{align}\label{cf_inj}
\Hom_R(E(R/\p),E(R/\p)^{(X)})\cong \widehat{R_\p^{(X)}}^\p.
\end{align}
Further, for any two indecomposable injective $R$-modules $E(R/\q)$ and $E(R/\p)$, there is an isomorphism:
\begin{align}\label{Hom_CF}
\Hom_R(E(R/\q),E(R/\p))\cong \begin{cases} \widehat{R_\q^{(X_\q)}}^\q,\text{ for some set $X_\q\not=\varnothing$,} & \q\subseteq \p\\ 0 & \q\not\subseteq \p\end{cases}.
\end{align}
For $\q\not\subseteq \p$, this is follows because $E(R/\p)\cong \Hom_{R}(R_\p,E(R/\p))$ and so adjointness yields an isomorphism
$$\Hom_R(E(R/\q),E(R/\p))\cong \Hom_{R}(R_\p\otimes_R E(R/\q),E(R/\q)),$$
which is $0$, as each element of $E(R/\q)$ is annihilated by a power of $\q$.

For $\q\subseteq \p$, one uses $E(R/\q)\cong R_\q\otimes_{R_\q} E(R/\q)$ and adjointness to show  
$$\Hom_R(E(R/\q), E(R/\p))\cong \Hom_{R_\q}(E(R/\q),\Hom_R(R_\q,E(R/\p))).$$
Since the module $\Hom_R(R_\q,E(R/\p))$ is an injective $R_\q$-module, it must have the form $\oplus_{\q'\subseteq \q}E(R/\q')$, but we have just shown that $\Hom_R(E(R/\q),E(R/\q'))=0$ if $\q'\subsetneq \q$ hence
$$\Hom_R(E(R/\q), \Hom_R(R_\q,E(R/\p)))\cong \Hom_R(E(R/\q),E(R/\q)^{(X)})$$
for some set $X$. Finally, apply the isomorphism in (\ref{cf_inj}).

\subsection*{Minimal complexes of cotorsion flat modules}
As defined in \cite{AM02}, we say a complex $C$ is {\em minimal} if every homotopy equivalence $\gamma:C\to C$ is an isomorphism. Similar to minimality criteria for injective resolutions or projective resolutions of finitely generated modules in a local ring, we have a criterion for minimality of complexes of cotorsion flat $R$-modules:
\begin{thm}\cite[Theorem 3.5]{Tho17a} \label{CF_minimal_thm}
Let $R$ be a commutative noetherian ring and $B$ a complex of cotorsion flat $R$-modules. The complex $B$ is minimal if and only if the complex $R/\p\otimes_R \Hom_R(R_\p,\Lambda^\p B)$ has zero differential for every $\p\in \Spec R$.
\end{thm}

\subsection*{Cotorsion flat resolutions and replacements}
For an $R$-module $M$, a {\em right cotorsion flat resolution} of $M$ is a complex $B$ of cotorsion flat $R$-modules along with a quasi-isomorphism $M\xrightarrow{\simeq} B$ such that $B^i=0$ for $i<0$.  Dually, a {\em left cotorsion flat resolution} of $M$ is a complex $B$ of cotorsion flat $R$-modules with $B^i=0$ for $i>0$ together with a quasi-isomorphism $B\xrightarrow{\simeq} M$.

Every flat $R$-module has a minimal right cotorsion flat resolution; every cotorsion $R$-module has a minimal left cotorsion flat resolution \cite[Theorem 5.2]{Tho17a}. Indeed every $R$-module $M$ is isomorphic in $\D(R)$ to a minimal semi-flat complex of cotorsion flat $R$-modules: there exists \cite[Theorem 5.2]{Tho17a} a diagram of quasi-isomorphisms
\begin{align}\label{CF_replacement}
B\xleftarrow{\simeq} F\xrightarrow{\simeq} M
\end{align}
where $F$ is a minimal flat resolution of $M$ (in fact, built from flat covers) and $B$ is a minimal semi-flat complex of cotorsion flat $R$-modules; we call $B$ a {\em minimal degreewise cotorsion flat replacement of $M$}. It turns out the slightly weaker notion of a minimal degreewise cotorsion flat replacement of $M$ (not necessarily a resolution) is sufficient for computing cosupport in Theorem \ref{cosupp_cf} below. It would be interesting to determine, given an $R$-complex $M$ (not just an $R$-module), whether one can find a minimal complex of cotorsion flat $R$-modules isomorphic to it in $\D(R)$.

\subsection*{Detecting cosupport}
We now show that a minimal degreewise cotorsion flat replacement of a module can be used to detect its cosupport, dual to the fact that minimal injective resolutions detect support \cite{Fox79}. This will be a primary tool in computing cosupport in the remainder of this paper.  
For a complex $B$ of cotorsion flat $R$-modules, we colloquially say $\p$ {\em appears in} $B$ if $\widehat{R_\p^{(X_\p)}}^\p$ is a nonzero summand of $B^i$ for some $i$. If $R$ is a commutative noetherian ring, $M$ is an $R$-module, and $B$ is a minimal semi-flat degreewise cotorsion flat replacement of $M$ (which exists by \cite[Theorem 5.2]{Tho17a}), then the following result shows that $\p\in \cosupp_R M$ if and only if $\p$ appears in $B$.

\begin{thm}\label{cosupp_cf}
Let $R$ be a commutative noetherian ring and $M$ be an $R$-complex. If there exists a minimal semi-flat complex of cotorsion flat $R$-modules $B$ which is isomorphic to $M$ in $\D(R)$, and one of the following holds: 
\begin{enumerate}
\item[(i)] $\pd_R R_\p<\infty$ for every $\p\in \Spec R$, or
\item[(ii)] $B$ is bounded on the left, that is, $B^i=0$ for all $i\ll 0$,
\end{enumerate}
then
$$\p\in \cosupp_R M \iff \smash{\widehat{R_\p^{(X)}}^\p\not=0} \text{ is a summand of $B^i$, some $i\in \Z$ and set $X$.}$$
\end{thm}
\begin{rem}\label{JRGthm}
If $R$ has finite Krull dimension, then work of Jensen \cite[Proposition 6]{Jen70} and Raynaud-Gruson \cite[Seconde partie, Th\`{e}or\'{e}me 3.2.6]{RG71} implies that $R$ satisfies condition (i) of the theorem, and we need no boundedness assumptions on $B$.  On the other hand, the minimal right cotorsion flat resolution of any flat $R$-module is semi-flat and bounded on the left, and so to compute $\cosupp_R R$, we need not impose any additional finiteness conditions on $R$.
\end{rem}

\begin{proof}[Proof of Theorem \ref{cosupp_cf}]
Since $B$ is semi-flat and isomorphic to $M$ in $\D(R)$, we have that $\p\in \cosupp_R M$ if and only if $\RHom_R(R_\p,\Lambda^\p B)\not\simeq0$ by (\ref{Lambda_semiflat}) and definition (\ref{cosupp_defn}). We will show that $\RHom_R(R_\p,\Lambda^\p B)\simeq \Hom_R(R_\p,\Lambda^\p B)$ and that $\Hom_R(R_\p,\Lambda^\p B)\not\simeq 0$ if and only if $\widehat{R_\p^{(X)}}^\p$ is a non-zero direct summand of $B^i$, for some $i\in \Z$ and some set $X$.

For each $i\in \Z$, the module $B^i$ has the form given by (\ref{CF_decomposition}), and so by (\ref{completion_and_colocalization_of_CF}) the complex $\Hom_R(R_\p,\Lambda^\p B)$ can be identified with the subquotient complex
$$\quad\quad \cdots \to \widehat{R_\p^{(X_\p^i)}}^\p \to \widehat{R_\p^{(X_\p^{i+1})}}^\p\to \cdots$$
of $B$, with induced differential.  It now follows that the following equivalences hold:
\begin{align*}
\text{$\widehat{R_\p^{(X)}}^\p$}&\text{ is a non-zero summand of $B^i$, for some $i\in \Z$ and set $X$,}\\
&\iff\Hom_R(R_\p,\Lambda^\p B)\not=0\\
&\iff R/\p\otimes_R \Hom_R(R_\p,\Lambda^\p B)\not=0\\
&\iff R/\p\otimes_R \Hom_R(R_\p,\Lambda^\p B)\not\simeq 0,
\end{align*}
where the second equivalence holds since $R/\p\otimes_R \widehat{R_\p^{(X_\p^i)}}^\p\cong (R_\p/\p R_\p)^{(X_\p^i)}$ for each $i\in \Z$ and the third equivalence holds because the complex $R/\p\otimes_R \Hom_R(R_\p,\Lambda^\p B)$ has zero differential by minimality of $B$ and Theorem \ref{CF_minimal_thm}.

To complete the proof, it remains to show that $R/\p \otimes_R \Hom_R(R_\p,\Lambda^\p B)\not\simeq 0$ if and only if $\Hom_R(R_\p,\Lambda^\p B)\not\simeq0$ and that $\RHom_R(R_\p,\Lambda^\p B)\simeq \Hom_R(R_\p,\Lambda^\p B)$. We accomplish these statments in the next two claims.

\medskip
\noindent
{\em Claim 1:} We have $\Hom_R(R_\p,\Lambda^\p B)\simeq0$ if and only if $R/\p\otimes_R \Hom_R(R_\p,\Lambda^\p B)\simeq 0$.\\

\noindent
{\em Proof of Claim 1:} 
Suppose $\Hom_R(R_\p,\Lambda^\p B)$ is acyclic and set $K^i=\ker(T_\p^i\to T_\p^{i+1})$, where $T_\p^i=\widehat{R_\p^{(X_\p^i)}}^\p$.  For each $i\in \Z$, the exact sequence $0\to K^i\to T_\p^i\to K^{i+1}\to 0$, along with minimality of $B$ (using Theorem \ref{CF_minimal_thm}), induces an exact sequence
$$R/\p\otimes_R K^i\twoheadrightarrow R/\p\otimes_R T_\p^i\xrightarrow{0}R/\p\otimes_R K^{i+1} \to 0.$$
It follows, for every $i\in \Z$, that $R/\p\otimes_R K^{i}=0$, and hence $R/\p\otimes_R T_\p^{i}=0$. Once again, as $R/\p\otimes_R-$ commutes with completion, we obtain that $X_\p^i=\varnothing$ and so $T_\p^{i}=0$ for every $i\in \Z$.  Therefore $\Hom_R(R_\p,\Lambda^\p B)=0$, and the forward implication follows.

Conversely, minimality of $B$ implies that $R/\p\otimes_R \Hom_R(R_\p,\Lambda^\p B)$ has zero differential by Theorem \ref{CF_minimal_thm}, and so acyclicity of this complex implies it is in fact the zero complex. Using again that $R/\p\otimes_RT_\p^i=0$ if and only if $T_\p^i=0$ for each $i\in \Z$, this forces $\Hom_R(R_\p,\Lambda^\p B)=0$. This justifies the first claim.

\medskip
\noindent
{\em Claim 2:} Assuming either condition (i) or (ii), there is an isomorphism in $\D(R)$:
\begin{align*}
\RHom_R(R_\p,\Lambda^\p B)\simeq\Hom_R(R_\p,\Lambda^\p B).
\end{align*}
{\em Proof of Claim 2:}
Let $f:P\xrightarrow{\simeq} R_\p$ be a projective resolution (chosen to be bounded if condition (i) holds) and set $C=\cone(f)$.  The triangulated functor $\Hom_R(-,\Lambda^\p B)$ on $\K(R)$ yields an exact triangle 
\begin{align}\label{triangle}
\Hom_R(C,\Lambda^\p B)\to\Hom_R(R_\p,\Lambda^\p B)\xrightarrow{f^*}\Hom_R(P,\Lambda^\p B)\to.
\end{align}
Since $C$ is an acyclic complex of flat $R$-modules that is bounded on the right, all kernels of $C$ are also flat, hence, because each module $(\Lambda^\p B)^i$ is cotorsion, the complex $\Hom_R(C,(\Lambda^\p B)^i)$ is acyclic for each $i$.  If either condition (i) or (ii) holds, it follows by \cite[Lemma 2.5]{CFH06} that $\Hom_R(C,\Lambda^\p B)$ is acyclic. Taking cohomology of (\ref{triangle}) shows that $f^*$ is a quasi-isomorphism.  Since $\RHom_R(R_\p,\Lambda^\p B)\simeq \Hom_R(P,\Lambda^\p B)$ in $\D(R)$, this verifies Claim 2.  
\end{proof}

An immediate consequence of Theorem \ref{cosupp_cf} is that we are able to easily construct a module with a prescribed cosupport. 
\begin{cor}
Let $W\subseteq \Spec R$ be any subset.  Then $M=\prod_{\p\in W}\widehat{R_\p}^\p$ is an $R$-module with $\cosupp_R M=W$.
\end{cor}

%%%%%%%%%%%%%%%%%%%%%%%%%%%%%%%%%%%%%%%%
\section{Cosupport of cotorsion modules} \label{section_cotorsion}
The purpose of this section is to illustrate how minimal cotorsion flat resolutions can be utilized to compute the cosupport of a cotorsion module, since every cotorsion module has such a resolution \cite[Theorem 5.2]{Tho17a}. 

For a ring $R$ and prime $\p\in\Spec R$, the module $\kappa(\p)=R_\p/\p R_\p$ is cotorsion, since $\kappa(\p) \cong \Hom_{R_\p}(\kappa(\p),E(R/\p))$.

\begin{prop}
Let $R$ be a commutative noetherian ring of finite Krull dimension.   Then
$$\cosupp_R\kappa(\p)=\{\p\}=\supp_R \kappa(\p).$$
\end{prop}
\begin{proof}
Localizing a minimal injective resolution of $R/\p$ at $\p$ yields a minimal injective resolution $\kappa(\p)\xrightarrow{\simeq} I$, which shows that $E(R/\p)$ is the only indecomposable injective $R$-module appearing in $I$; in particular, $\supp_R\kappa(\p)=\{\p\}$.  Further, one obtains a resolution:
$$\Hom_R(I,E(R/\p))\xrightarrow{\simeq} \Hom_R(\kappa(\p),E(R/\p))\cong \kappa(\p).$$
Since $E(R/\p)$ is the only indecomposable injective module in $I$, the complex
$$\Hom_R(I,E(R/\p))\quad=\quad \cdots \to \widehat{R_\p^{(X_\p^1)}}\to \widehat{R_\p^{(X_\p^0)}}\to 0$$
is a left cotorsion flat resolution of $\kappa(\p)$ with $X_\q^i=0$ for all $\q\not=\p$ and $X_\p^0\not=0$.  We claim $\Hom_R(I,E(R/\p))$ is minimal: Since $I$ is minimal, $\Hom_{R_\p}(\kappa(\p),I_\p)$ has zero differential, and therefore, by standard adjointness and \cite[Proposition 2.1(ii)]{CH09}:
\begin{align*}
R/\p\otimes_R \Hom_R(R_\p,\Hom_R(I,E(R/\p))) &\cong \kappa(\p)\otimes_{R_\p}\Hom_{R_\p}(I_\p,E(R/\p))\\
&\cong \Hom_{R_\p}(\Hom_{R_\p}(\kappa(\p),I_\p),E(R/\p))
\end{align*}
has zero differential as well, implying that $\Hom_R(I,E(R/\p))$ is minimal by Theorem \ref{CF_minimal_thm}.  Finally, the result follows by Theorem \ref{cosupp_cf}.
\end{proof}

Cosupport of an injective $R$-module has been described elsewhere, see \cite[Proposition 5.4]{BIK12} and \cite[Proposition 6.3]{SWW17}, but we give a different proof here to illustrate the use of minimal complexes of cotorsion flat modules to compute cosupport. 
\begin{prop}\label{cosupp_of_indec_inj}
Let $R$ be a commutative noetherian ring of finite Krull dimension and $E(R/\p)$ an indecomposable injective $R$-module. Then
$$\cosupp_RE(R/\p)= \{\q\in\Spec R\mid \q\subseteq \p\}.$$
\end{prop}
\begin{proof}
Let $R\xrightarrow{\simeq} I$ be the minimal injective resolution of $R$.  Then there is a resolution
$$\Hom_R(I,E(R/\p))\xrightarrow{\simeq} \Hom_R(R,E(R/\p))\cong E(R/\p).$$
This shows that $\Hom_R(I,E(R/\p))$ is a left cotorsion flat resolution of $E(R/\p)$; our goal is to apply Theorem \ref{CF_minimal_thm} to show it is minimal.

For each prime ideal $\q$ of $R$, standard adjointness and \cite[Proposition 2.1(ii)]{CH09} yields the following isomorphisms:
\begin{align*}
R/\q\otimes_R \Hom_R(R_\q,\Hom_R(I,E(R/\p)))&\cong R/\q\otimes_R \Hom_R(I_\q,E(R/\p))\\
&\cong \Hom_R(\Hom_R(R/\q,I_\q), E(R/\p)).
\end{align*}
Since $I$ is minimal, the complex $\Hom_R(R/\q,I_\q)$ ($\cong \Hom_{R_\q}(\kappa(\q),I_\q)$) has zero differential, and it follows that the complex
$$R/\q\otimes_R \Hom_R(R_\q,\Hom_R(I,E(R/\p)))$$
has zero differential as well. There is a degreewise isomorphism of complexes $\Hom_R(I,E(R/\p))\cong \Lambda^\p\Hom_R(I,E(R/\p))$, by the isomorphism (\ref{Hom_CF}), and it now follows from Theorem \ref{CF_minimal_thm} that $\Hom_R(I,E(R/\p))$ is a minimal left cotorsion flat resolution of $E(R/\p)$. 

For every prime ideal $\q$ in $R$, the indecomposable injective $R$-module $E(R/\q)$ appears in $I$, and so (\ref{Hom_CF}) shows that for every $\q\subseteq \p$, there is a nonempty set $X$ and integer $i$ such that $\widehat{R_\q^{(X)}}^\q$ appears as a nonzero summand of $\Hom_R(I,E(R/\p))^i$. The claim regarding cosupport now follows from Theorem \ref{cosupp_cf}.
\end{proof}

%%%%%%%%%%%%%%%%%%%%%%%%%%%%%%%%%%%%%%%%
\section{Cosupport of finitely generated modules}\label{cosupp_of_fg_modules}
We turn our focus to computing cosupport of finitely generated modules (or complexes with degreewise finitely generated cohomology). To do so, we will employ the following fact about $\cosupp_R M\otimes_R^\L N$, which complements the corresponding fact \cite[Theorem 9.7]{BIK12} that $\cosupp_R\RHom_R(M,N)=\supp_RM\cap \cosupp_RN$ for any $R$-complexes $M$ and $N$. 

We first prove a lemma regarding when evaluation morphisms are invertible in $\D(R)$; more cases for when these maps are invertible can be found in the literature (see for example \cite[Proposition 2.2]{CH09}), we only include statements (which appear not to be recorded elsewhere) that focus on one of the complexes having degreewise finitely generated cohomology without boundedness restrictions.

\begin{lem}\label{evaluation_maps}
Let $R$ be a commutative noetherian ring, and $X$, $Y$, and $Z$ be $R$-complexes. There are canonical $R$-linear evaluation morphisms:
\begin{align*}
\omega_{{}_{XYZ}}:\ &\RHom_R(X,Y)\otimes_R^\L Z\to \RHom_R(X,Y\otimes_R^\L Z),\\
\theta_{{}_{XYZ}}:\ &X\otimes_R^\L \RHom_R(Y,Z)\to \RHom_R(\RHom(X,Y),Z).
\end{align*}
The tensor evaluation morphism $\omega_{{}_{XYZ}}$ is a quasi-isomorphism provided either:
\begin{itemize}
\item[(1)] The complex $X$ is isomorphic in $\D(R)$ to a bounded complex of projective $R$-modules, the complex $Y$ is isomorphic in $\D(R)$ to a bounded complex of flat $R$-modules, and $\H^i(Z)$ is finitely generated for each $i\in \Z$.
\item[(2)] For each $i\in \Z$, the module $\H^i(X)$ is finitely generated, the complex $Y$ is isomorphic in $\D(R)$ to a bounded complex of injective $R$-modules, and the complex $Z$ is isomorphic in $\D(R)$ to a bounded complex of flat $R$-modules.
\end{itemize}
The Hom evaluation morphism $\theta_{{}_{XYZ}}$ is a quasi-isomorphism provided that:
\begin{itemize}
\item[(3)] For each $i\in \Z$, the module $\H^i(X)$ is finitely generated, and the complexes $Y$ and $Z$ are each isomorphic in $\D(R)$ to bounded complexes of injective $R$-modules.
\end{itemize}
\end{lem}
\begin{proof}
$(1)$: For any projective $R$-module $P$ and flat $R$-module $F$, the $R$-module $\Hom_R(P,F)$ is flat; it follows that the complex $\RHom_R(X,Y)$ is isomorphic in $\D(R)$ to a bounded complex of flat $R$-modules, and therefore the functors 
$$G=\RHom_R(X,Y)\otimes_R^\L - \quad \text{ and }\quad G'=\RHom_R(X,Y\otimes_R^\L -)$$ 
on $\D(R)$ are way-out functors in the sense of \cite[I, section 7]{Har66}, that is, the functors preserve bounded cohomology. There is a natural transformation $\eta:G\to G'$ determined by $\omega_{{}_{XYZ}}$, and for each finitely generated $R$-module $M$, the map $\eta(M)$ is a quasi-isomorphism by \cite[Proposition 2.2(vi)]{CH09}.  Since $G$ and $G'$ are way-out functors, we obtain by \cite[I, Proposition 7.1(iv)]{Har66} that $\eta(Z)$, and hence $\omega_{{}_{XYZ}}$, is a quasi-isomorphism for all complexes $Z$ with $\H^i(Z)$ finitely generated for all $i\in \Z$.

$(2)$ and $(3)$ follow similarly by way-out techniques \cite[I, Proposition 7.1(iv)]{Har66}, along with \cite[Proposition 2.2(iv)]{CH09} and \cite[Proposition 2.2(ii)]{CH09}, respectively.
\end{proof}

\begin{prop}\label{cosupp_tensor}
Let $R$ be a commutative noetherian ring and $M$ and $N$ be $R$-complexes.  Suppose one of the following holds:
\begin{enumerate}
\item $\pd_R R_\p<\infty$ for every $\p\in \Spec R$, $\H^i(M)$ is finitely generated for each $i$, and $N$ is isomorphic in $\D(R)$ to a bounded complex of flat modules, or
\item $\pd_R R_\p<\infty$ for every $\p\in \Spec R$, $\H^i(M)$ is finitely generated for each $i$, $\H^i(M)=0$ for $i\gg0$, and $N$ has bounded cohomology, or
\item $M$ is isomorphic in $\D(R)$ to a bounded complex of finitely generated projective $R$-modules and $N$ has bounded cohomology.
\end{enumerate}
Then 
$$\cosupp_RM\otimes_R^\L N= \supp_R M\cap \cosupp_RN.$$
\end{prop}
\begin{proof}
Fix $\p\in \Spec R$ and consider the natural tensor evaluation map
\begin{align}\label{hom_eval}
\omega:\RHom_R(R_\p,N)\otimes_R^\L M\to \RHom_R(R_\p,N\otimes_R^\L M).
\end{align}
Assuming condition (1), the map $\omega$ is a quasi-isomorphism by Lemma \ref{evaluation_maps}(1); assuming (2) or (3), it is a quasi-isomorphism by \cite[Proposition 2.2(vi,v)]{CH09}.

We have that $\p\in \cosupp_R M\otimes_R^\L N$ if and only if the following hold, using the equivalent descriptions of cosupport in (\ref{equiv_defns_of_cosupp}):
\begin{align*}
\RHom_R(R_\p,N&\otimes_R^\L M)\otimes_R^\L \kappa(\p) \not\simeq 0 \\
&\iff \RHom_R(R_\p,N)\otimes_R^\L M\otimes_R^\L \kappa(\p)\not\simeq0\text{, by (\ref{hom_eval}),}\\
&\iff (\RHom_R(R_\p,N)\otimes_R^\L\kappa(\p))\otimes_{\kappa(\p)}(M\otimes_R^\L \kappa(\p))\not\simeq 0\\
&\iff \p\in  \cosupp_R N\cap \supp_R M,
\end{align*}
where the last equivalence uses the K\"{u}nneth formula \cite[Theorem 3.6.3]{Wei94}.
\end{proof}

Part (1) of the following corollary extends \cite[Theorem 6.6]{SWW17} to unbounded complexes and some rings without dualizing complexes, including all rings of finite Krull dimension. Recall that an $R$-complex is called {\em perfect} if it is isomorphic in $\D(R)$ to a bounded complex of finitely generated projective $R$-modules.
\begin{cor}\label{cosupp_vs_supp}
Let $M$ be an $R$-complex with degreewise finitely generated cohomology. If at least one of the following holds:
\begin{enumerate}
\item $\pd_RR_\p<\infty$ for every $\p\in \Spec R$, or
\item $M$ is a perfect $R$-complex,
\end{enumerate}
then
$$\cosupp_RM=\supp_RM\cap \cosupp_RR.$$
\end{cor}

\begin{proof}
Apply Proposition \ref{cosupp_tensor} with $N=R$.
\end{proof}

The corollary puts emphasis on computing the cosupport of $R$. Recall \cite[Theorem 5.2]{Tho17a} that $R$ has a minimal right cotorsion flat resolution; indeed, the {\em minimal pure-injective resolution}\footnote{Minimal pure-injective resolutions were referred to as right $\PurInj$-resolutions in \cite{Tho17a}.} of $R$ (built from pure-injective envelopes; see \cite{EJ00} for details) is such a resolution.  This allows us to invoke Enochs' description \cite{Eno87} of minimal pure-injective resolutions in order to determine $\cosupp_RR$.  

\begin{rem}
Although Theorem \ref{cosupp_cf} allows us to compute $\cosupp_RR$ without any finiteness conditions on $R$, it would be interesting to determine whether the conclusions of Proposition \ref{cosupp_tensor} and Corollary \ref{cosupp_vs_supp} hold without the assumption that $\pd_R R_\p<\infty$ for every $\p\in \Spec R$.  See also Example \ref{full_cosupp_examples}(3) below.
\end{rem}

The following change of rings result allows us to compare cosupport along finite ring maps, by understanding cotorsion flat modules under finite base change (cf. \cite[Theorem 7.11]{BIK12}). A ring homomorphism $R\to S$ is called {\em finite} if $S$ is finitely generated as an $R$-module; in addition, to every ring map $f:R\to S$ we can associate a map $f^*:\Spec S\to \Spec R$ defined by sending a prime ideal $\p\subseteq S$ to its contraction $f^{-1}(\p)\subseteq R$, which is necessarily prime as well. 
\begin{thm}\label{finite_map_cosupp}
If $f:R\to S$ is a finite map of commutative noetherian rings, then  
$$S\otimes_R\left(\prod_{\p\in \Spec R}\widehat{R_\p^{(X_\p)}}^\p\right)\cong  \prod_{\q\in \Spec S}\widehat{S_\q^{(X_{\q})}}^\q\text{, where $X_\q=X_\p$ for $f^*(\q)=\p$.}$$ 
Consequently, 
\begin{align*}
\cosupp_SS&=(f^*)^{-1}(\cosupp_R R),
\end{align*}
or in other words, for $\q\in \Spec S$, $\q\in \cosupp_SS$ if and only if $f^*(\q)\in \cosupp_RR$.
\end{thm}
\begin{proof}
For a prime $\p\in \Spec R$, the set $(f^*)^{-1}(\p)=\{\q\in \Spec S\mid f^*(\q)=\p\}$ is finite. Fix $\p\in \Spec R$ and set $W=(f^*)^{-1}(\p)$.  We will show
\begin{align}\label{CF_under_base_change}
S\otimes_R\widehat{R_\p^{(X_\p)}}^\p\cong \bigoplus_{\q\in W} \widehat{S_\q^{(X_\p)}}^\q.
\end{align}
This is enough to establish the first claim, as $S$ is finitely generated as an $R$-module.  The assertion regarding cosupport then follows from Theorem \ref{cosupp_cf} applied to a minimal pure-injective resolution of $R$, as follows: Let $R\xrightarrow{\simeq} B$ be a minimal pure-injective resolution of $R$ (i.e., a right resolution built from pure-injective envelopes). Applying $S\otimes_R-$ preserves pure-injective envelopes because $S$ is finitely generated as an $R$-module, so that $S\xrightarrow{\simeq} S\otimes_RB$ is a minimal pure-injective resolution of $S$ \cite[Theorem 4.2]{Eno87}. By \cite[Theorem 5.2]{Tho17a}, $B$ and $S\otimes_RB$ are minimal right cotorsion flat resolutions of $R$ and $S$, respectively. By Theorem \ref{cosupp_cf}, the primes appearing in $B$ are precisely those in $\cosupp_RR$ and the primes appearing in $S\otimes_RB$ are those in $\cosupp_SS$. The statement relating the cosupport of $R$ and $S$ now follows once we have verified (\ref{CF_under_base_change}).

To establish (\ref{CF_under_base_change}), we recall the following fact \cite[Theorem 1.1]{Rah09}:
\begin{align}\label{injectives_along_finite_maps}
\Hom_R(S,E_R(R/\p))\cong \bigoplus_{\q\in W}E_S(S/\q).
\end{align}
With this in hand, we apply $S\otimes_R -$ to the cotorsion flat module $\widehat{R_\p^{(X_\p)}}^\p$, using that $S$ is finitely generated over $R$ so that the second isomorphism below follows from \cite[Proposition 2.1(ii)]{CH09} and the third isomorphism below is by standard adjunction along with the fact that $\Hom_R(S,-)$ commutes with arbitrary direct sums:
\begin{align*}
S\otimes_R \widehat{R_\p^{(X_\p)}}^\p&\cong S\otimes_R \Hom_R(E(R/\p),E(R/\p)^{(X_\p)})\text{, by \cite[Lemma 4.1.5]{Xu96},}\\
&\cong \Hom_R(\Hom_R(S,E(R/\p)), E(R/\p)^{(X_\p)}),\\
&\cong \Hom_S(\Hom_R(S,E(R/\p)),\Hom_R(S,E(R/\p))^{(X_\p)}),\\
&\cong \Hom_S(\bigoplus_{\q\in W}E_S(S/\q),\bigoplus_{\q\in W}E_S(S/\q)^{(X_\p)})\text{, by (\ref{injectives_along_finite_maps}).}
\end{align*}
Finally, letting $\q',\q''\in W$, we claim that $\Hom_S(E_S(S/\q'),E_S(S/\q''))=0$ whenever $\q'\not=\q''$. First, by \cite[Corollary 5.9]{AM69} we notice that $\q'$ cannot be strictly contained in $\q''$ by the definition of $W$. On the other hand, if $\q'\not\subseteq \q''$, then as $(E_S(S/\q'))_{\q''}=0$ but $E_S(S/\q'')^{(X_{\p})}$ is $\q''$-local, i.e., $E_S(S/\q'')^{(X_{\p})}\cong (E_S(S/\q'')^{(X_{\p})})_{\q''}$, one obtains the claim by standard adjunction. The previous display now yields the following:
\begin{align*}
S\otimes_R \widehat{R_\p^{(X_\p)}}^\p&\cong \bigoplus_{\q\in W}\Hom_S(E_S(S/\q),E_S(S/\q)^{(X_\p)})\cong \bigoplus_{\q\in W}\widehat{S_\q^{(X_\p)}}^\q,
\end{align*}
where we apply \cite[Lemma 4.1.5]{Xu96} to obtain the last isomorphism.
\end{proof}

We immediately obtain:
\begin{cor}\label{immediate_consequences}
If $R\to S$ is a finite map of commutative noetherian rings, and $\cosupp_RR=\Spec R$, then $\cosupp_SS=\Spec S$. 

Furthermore, if the map $\pi: R\twoheadrightarrow S$ is surjective, then for $\p\supseteq \ker(\pi)$, we have $\pi(\p)\in \cosupp_SS$ if and only if $\p\in \cosupp_RR$. \hfill $\square$
\end{cor}

\begin{rem}\label{max_in_cosupp}
This result recovers the fact\footnote{In fact, $\max(\cosupp_RM)=\max(\supp_RM)$ for any $R$-complex $M$ \cite[Theorem 4.13]{BIK12}.} that if $R$ is a commutative noetherian ring and $\m$ is a maximal ideal, then $\m\in \cosupp_RR$: From the finite map $\pi:R\twoheadrightarrow R/\m$, we see that since $0\in \cosupp_{R/\m}R/\m$ and $\pi^*(0)=\m$, that $\m\in \cosupp_RR$.
\end{rem}

Recall that for any commutative noetherian ring $R$ we have\footnote{This can also be seen from the minimal pure-injective resolution of $R$, by \cite[Corollary 2.6]{Eno95}.} the following inclusion \cite[Proposition 4.19]{BIK12}: 
\begin{align}\label{cosupp_conj_equality}
\cosupp_RR\subseteq \bigcap_{\text{$R$ is $\a$-complete}} \cV(\a),
\end{align}
where $\cV(\a)=\{\p\in \Spec R\mid \p\supseteq\a\}$. Set $\cc_R=\sum \a$, with the sum over all ideals $\a$ such that $R$ is $\a$-complete. Note that $R$ is $\cc_R$-complete and if $\b\supsetneq \cc_R$, then $R$ is not $\b$-complete. There is an equality $\bigcap \cV(\a)=\cV(\cc_R )$, where the intersection is taken over all ideals $\a$ such that $R$ is $\a$-complete.  

One of our goals is to investigate when the inclusion $\cosupp_RR\subseteq \cV(\cc_R )$ is an equality; in particular, we show that equality holds for any ring $R$ such that $R/\cc_R $ is countable (using Theorem \ref{full_cosupp_rings} below). 

\begin{prop}\label{cosupp_complete_closed_thm}
Let $R$ be a commutative noetherian ring and let $\cc_R$ be defined as above. 
The following are equivalent:
\begin{enumerate}
\item Equality in (\ref{cosupp_conj_equality}) holds; i.e., $\cosupp_RR=\cV(\cc_R )$;
\item $R/\cc_R$ has full cosupport; i.e., $\cosupp_{R/\cc_R}(R/\cc_R)=\Spec(R/\cc_R)$;
\item For every $\p\in \cV(\cc_R)$, $\Ext_{R/\p}^i((R/\p)_{(0)},R/\p)\not=0$ for some $i$.
\end{enumerate}
\end{prop}
\begin{proof}
For any ideal $I\subseteq R$ and $\p\in \cV(I)$, we have $\p/I\in \cosupp_{R/I}R/I$ if and only if $\p\in \cosupp_RR$ by Corollary (\ref{immediate_consequences}). In conjunction with the inclusion (\ref{cosupp_conj_equality}), the equivalence of (1) and (2) then follows for $I=\cc_R$.  Moreover, for $\p\in \cV(\cc_R)$, $\p\in \cosupp_RR$ if and only if $0\in \cosupp_{R/\p}R/\p$ again by Corollary (\ref{immediate_consequences}), and hence (1) is equivalent to (3) by definition (\ref{cosupp_for_flat_module_defn}). 
\end{proof}

We caution that equality in (\ref{cosupp_conj_equality}) need not hold in general; see Example \ref{cosupp_not_closed} below. 
Indeed, equality need not hold even for noetherian domains of finite Krull dimension that are only 0-complete (i.e., not complete at any nonzero ideal). 

Part (3) of the following result gives an affirmative answer to part of the question \cite[Question 6.16]{SWW17} and part (2) avoids the assumption of a dualizing complex of \cite[Theorem 6.11]{SWW17}. This result also displays some of the subtleties of cosupport; indeed, there are rings of any Krull dimension having full cosupport, see part (1), and also rings of any cardinality having full cosupport, see part (3). Part (2), along with Corollary \ref{cosupp_equal_supp} below, recovers \cite[Proposition 4.18]{BIK12} and \cite[Theorem 6.11]{SWW17}. 

\begin{thm}\label{full_cosupp_rings}
If $R$ is one of the following rings, then $\cosupp_RR=\Spec R$.
\begin{enumerate}
\item A countable commutative noetherian ring;
\item A 1-dimensional commutative noetherian domain, not complete local; 
\item The ring $k[x,y]_{(x,y)}$, for any field $k$.
\end{enumerate}
Moreover, if $R\to T$ is a finite ring map, then $T$ also satisfies $\cosupp_TT=\Spec T$.
\end{thm}

\begin{proof}
For finite ring maps $R\to T$, Theorem \ref{finite_map_cosupp} shows that if $\cosupp_RR=\Spec R$, then $\cosupp_TT=\Spec T$.  

To address (1), let $R$ be any countable commutative noetherian ring.  For a prime ideal $\p\in \Spec R$, Theorem \ref{finite_map_cosupp} shows that $\p\in \cosupp_RR$ if and only if $0\in \cosupp_{R/\p}R/\p$.  Therefore, it is sufficient to assume $R$ is a countable domain and show $0\in \cosupp_RR$. 

If $R$ is a field, $R$ trivially has full cosupport.  It is therefore enough to consider the case where $R$ is not a field, in which case there exists a short exact sequence (see \cite[(3.1)]{Tho15})
$$0\to R\to \invlim_{s\in S} R/s R\to \Ext_R^1(R_{(0)},R)\to 0,$$
where $S=R\setminus \{0\}$ is pre-ordered by divisibility: $s'|s$ if and only if  $sR\subseteq s'R$.  In this case, $\invlim_{s\in S} R/s R$ is uncountable, and so the first map in this short exact sequence is not surjective, hence $\Ext_R^1(R_{(0)},R)\not=0$; see also \cite[Setup 3]{Tho15}.  By definition (\ref{cosupp_for_flat_module_defn}), we have $0\in\cosupp_RR$.  It follows that rings as in (1) have full cosupport. 

Next, if $R$ is as in (2), then since $R$ has Krull dimension $1$, the minimal pure-injective resolution of $R$ has the form \cite{Eno87}:
$$B:=\quad 0\to \prod_{\m\text{ maximal}} \widehat{R}^\m\to T_0\to 0,$$
where $T_0=\widehat{R_{(0)}^{(X)}}$ for some set $X$. As $R$ is a domain that is not complete local, we must have $T_0\not=0$. Since $B$ is a minimal (semi-flat) right cotorsion flat resolution of $R$ \cite[Theorem 5.2]{Tho17a}, Theorem \ref{cosupp_cf} yields that $\cosupp_RR=\Spec R$.

For (3), if $k$ is countable, the result follows from (1), so assume $k$ is uncountable. In this case, the ring $R=k[x,y]_{(x,y)}$ satisfies $\Ext_R^2(R_{(0)},R)\not=0$ by \cite[Proposition 3.2]{Gru71}.  Thus $0\in \cosupp_RR$. For $0\not=\p\in \Spec R$, the ring $R/\p$ is either a field or a ring as in (2), and so has full cosupport. Applying Theorem \ref{finite_map_cosupp} to the map $R\to R/\p$ for each $\p\not=0$, we obtain that $\cosupp_RR=\Spec R$.

\end{proof}

We conclude that cosupport and support coincide for any complex with degreewise finitely generated cohomology over any of the rings in Theorem \ref{full_cosupp_rings}. 

\begin{cor}\label{cosupp_equal_supp}
Let $R$ be any ring as in Theorem \ref{full_cosupp_rings} which also satisfies the condition that $\pd_R R_\p<\infty$ for every $\p\in \Spec R$, and let $M$ be an $R$-complex with degreewise finitely generated cohomology.  Then
$$\cosupp_R M=\supp_R M.$$
\end{cor}
\begin{proof}
By Corollary \ref{cosupp_vs_supp} and Theorem \ref{full_cosupp_rings}, we have 
$$\cosupp_R M =\cosupp_R R \cap \supp_R M =\supp_R M.$$
\end{proof}

\begin{cor}\label{cosupp_V_countable}
If $R/\cc_R $ is one of the rings in Theorem \ref{full_cosupp_rings}, then
$$\cosupp_RR=\cV(\cc_R ).$$
\end{cor}
\begin{proof}
Combine Proposition \ref{cosupp_complete_closed_thm} and Theorem \ref{full_cosupp_rings}.
\end{proof}

A conjecture, initiated in the early 1970s by Gruson \cite{Gru71} and Jensen \cite{Jen72}, and then generalized by Gruson in 2013 and formalized by Thorup \cite{Tho15}, allows us to conjecture that rings having full cosupport are far more ubiquitous.  This conjecture states the following: For a field $k$ and integer $n\geq 0$, let $R=k[x_1,...,x_n]$ be the polynomial ring in $n$ variables.  Set $c=0$ if $k$ is finite and define $c$ by the cardinality $|k|=\aleph_c$ if $k$ is infinite.  With this setup, it is conjectured that:
\begin{align}\label{Gruson_conj}
\Ext_R^i(R_{(0)},R)\not=0 \iff i=\inf \{c+1,n\}.
\end{align}
This conjecture is true when $k$ is at most countable or $n\leq 1$; e.g., see \cite{Tho15}. 
If this conjecture were true in general, it would follow from Noether normalization that a commutative noetherian ring $S$ which is finitely generated as a $k$-algebra would have full cosupport by Theorem \ref{finite_map_cosupp}.

%%%%%%%%%%%%%%%%%%%%%%%%%%%%%%%%%%%%%%%%
\section{Examples of $\cosupp_RR$}\label{examples_section}
The following question is motivated by Proposition \ref{cosupp_complete_closed_thm}. 

\begin{question}
When do $0$-complete noetherian domains have full cosupport? 
\end{question}

The examples below illustrate the nuances of this question.  We start with a warm-up of some examples of rings having full cosupport.
\begin{example}\label{full_cosupp_examples}
The following rings $R$ satisfy $\cosupp_RR=\Spec R$:
\begin{enumerate}
\item Let $k$ be a countable field and $R=k[x_1,...,x_n]/\a$, for $n\geq 0$ and an ideal $\a$;
\item Let $k$ be an uncountable field and $R=k[x_1,x_2]_{(x_1,x_2)}/\a$, for an ideal $\a$;
\item Let $R$ be Nagata's example \cite[Appendix, Example 1]{Nag62} of a commutative noetherian ring of infinite Krull dimension, under the additional assumption that the coefficient field is countable;
\item Let $p$ be a prime number and $R=\Z_{(p)}$ be the localization of $\Z$ at the prime ideal $(p)$; more generally, let $R$ be a discrete valuation ring which is not complete at its maximal ideal.
\end{enumerate}
These all follow immediately from Theorem \ref{full_cosupp_rings}: (1) and (3) are countable, (2) is a finite ring extension of $k[x,y]_{(x,y)}$, and (4) is dimension 1 and not complete local.
\end{example}

For contrast, recall that the cosupport of $R$ can fall short of $\Spec R$; in particular, the cosupport of a complete semi-local ring is the set of maximal ideals (cf. \cite[Proposition 4.19]{BIK12}):
\begin{example}\label{semi_local_ring}
Let $R$ be a complete semi-local ring, that is, a ring with finitely many maximal ideals $\m_1,...,\m_n$ that is complete at the Jacobson radical $\bigcap_{i=1}^n\m_i$.  The minimal right cotorsion flat resolution has one term: $\prod_{i=1}^n \widehat{R_{\m_i}}^{\m_i}$.  Theorem \ref{cosupp_cf} implies that $\cosupp_R R=\{\m_1,...,\m_n\}.$ 
In particular, a complete local ring $(R,\m)$ has cosupport equal to $\{\m\}$.
\end{example}

In order to understand how cosupport behaves under adjoining power series variables, we first prove:
 
\begin{prop}\label{cosupp_bij_mod_complete_ideal}
Let $S$ be a ring that is $I$-complete. Then the canonical surjection $\pi:S\twoheadrightarrow S/I$ induces a homeomorphism of topological spaces:
$$\pi^*:\cosupp_{S/I}(S/I) \xrightarrow{\cong} \cosupp_SS.$$
\end{prop}
\begin{proof}
The natural surjection $\pi:S\twoheadrightarrow S/I$ induces a homeomorphism of topological spaces (i.e., a continuous bijection whose inverse is also continuous) $\pi^*:\Spec(S/I)\xrightarrow{\cong} \cV(I)\subseteq \Spec S$ \cite[Chapter 1, Exercise 21]{AM69}. Since $S$ is $I$-complete, $\cosupp_SS\subseteq \cV(I)$. For $\p\in \Spec(S/I)$, Theorem \ref{finite_map_cosupp} implies $\p\in\cosupp_{S/I}(S/I)$ if and only if $\pi^*(\p)\in \cosupp_SS$. Hence $\pi^*$ induces a bijection between $\cosupp_{S/I}(S/I)$ and $\cosupp_SS$. Indeed, endowing $\cosupp_{S/I}(S/I)\subseteq \Spec(S/I)$ and $\cosupp_SS\subseteq \cV(I)$ each with the subspace topology, we obtain that $\pi^*$ induces the desired homeomorphism.
\end{proof}

\begin{example}\label{power_series_example}
If $R$ is any ring and $S=R[\![t_1,...,t_n]\!]$ for $n\geq 0$, then Proposition \ref{cosupp_bij_mod_complete_ideal} yields a homeomorphism 
$$\cosupp_RR\xrightarrow{\cong} \cosupp_SS,$$
using that $S$ is $(t_1,...,t_n)$-complete \cite[Exercise 8.6]{Mat89}. 
In particular, if $k$ is a field and $S=k[x][\![t]\!]$, then $\cosupp_SS=\cV((t))$. 
\end{example}

The next example we give shows that the cosupport of $R$ need not be a closed subset of $\Spec R$, i.e., there are rings $R$ such that $\cosupp_RR\not=\cV(I)$ for any ideal $I$. In particular, it shows that we can have a strict inequality $\cosupp_RR\subsetneq\cV(\cc_R )$. This provides a negative answer to the question \cite[Question 6.13]{SWW17}. 

\begin{example}\label{cosupp_not_closed}
Let $k$ be a field and set $T=k[\![t]\!][x]$.  Applying Theorem \ref{finite_map_cosupp} to the finite map $T\twoheadrightarrow T/(x)\cong k[\![t]\!]$, the fact that $0\not\in \cosupp_{k[\![t]\!]}k[\![t]\!]$ (see Example \ref{semi_local_ring}) implies that $(x)\not\in \cosupp_TT$, so that $\cosupp_TT\subsetneq \cV((0))$. 

The ring $T$ has uncountably many height 1 prime ideals that are maximal \cite[Theorem 3.1, Remarks 3.2]{HRW06}, even if $k$ is finite.  Let $\cP$ be the set of all height 1 maximal ideals\footnote{For our purposes, we only need $\cP$ to be an infinite set, and we may take $\cP=\{(1-xt^n)\}_{n\geq 1}$.  To see that for each $n\geq 1$, $\p_n:=(1-xt^n)$ is a maximal ideal, just observe that every nonzero element of $T/\p_n$ is a unit; this follows because the images of $x$ and $t$ are both units.}.  As $\p\in \cP$ are maximal, Remark \ref{max_in_cosupp} implies that $\p\in \cosupp_TT$. 

Since $T$ is a noetherian unique factorization domain, every height 1 prime ideal is principal \cite[Theorem 20.1]{Mat89}.  
For each $\p\in \cP$, we may write $\p=(f_\p)$, for a prime element $f_\p\in T$.  Define the ideal $I=\bigcap_{\p\in \cP}(f_\p)$.  If $\alpha\in I$, then $\alpha$ must be divisible by $f_\p$ for every $\p\in \cP$, forcing $\alpha=0$ since $T$ is a unique factorization domain.  Therefore $I=0$. If $\cosupp_TT\subseteq \cV(J)$ for some ideal $J\subseteq T$, then $\p\supseteq J$ for every $\p\in \cP$, hence $0=I\supseteq J$, i.e., $J=0$. However, $\cosupp_TT\not=\cV((0))$, hence it is not a closed subset of $\Spec T$ and we have a strict containment $\cosupp_TT\subsetneq \cV(\cc_T)$. 
\end{example}

\begin{example}
Let $k$ be any field and set $S=k[\![t]\!][x][\![s_1,...,s_n]\!]$, for $n\geq 0$. We claim that $\cosupp_SS$ is not a closed subset of $\Spec S$. There is a canonical surjection $\pi:S\twoheadrightarrow T$, where $T$ is the ring from Example \ref{cosupp_not_closed}, which induces a homeomorphism $\pi^*:\Spec T\xrightarrow{\cong} \cV((s_1,...,s_n))$. As $S$ is $(s_1,...,s_n)$-complete, $\cosupp_SS\subseteq \cV((s_1,...,s_n))$, and so Proposition \ref{cosupp_bij_mod_complete_ideal} shows that $(\pi^*)^{-1}(\cosupp_SS)=\cosupp_TT$.  As $\pi^*$ is continuous and $\cosupp_TT$ is not closed, $\cosupp_SS$ cannot be closed in $\Spec S$.  This yields a family of rings without closed cosupport.
\end{example}

\noindent
{\bf Acknowledgements:} This paper began as work in my dissertation at the University of Nebraska-Lincoln.  I am profoundly grateful to my advisor, Mark Walker, who was an indispensable source of advice and support in this work. I would also like to thank Lars Winther Christensen, Douglas Dailey, and Thomas Marley for helpful conversations, as well as the support I have received at Texas Tech University. I am also grateful to the anonymous referee for many helpful suggestions.

%% Bibliography goes here 
%% BibTeX is your friend
%\bibliographystyle{plain}
%\bibliography{refs}

\end{document}